\documentclass[english,a4paper]{scrartcl}
\usepackage[T1]{fontenc}
\usepackage[latin9]{inputenc}
\usepackage[a4paper]{geometry}
\usepackage{babel}
\usepackage{verbatim}
\usepackage{prettyref}
\usepackage{mathtools}
\usepackage{amsmath}
\usepackage{amsthm}
\usepackage{amssymb}
\usepackage[unicode=true,pdfusetitle,
 bookmarks=true,bookmarksnumbered=false,bookmarksopen=false,
 breaklinks=false,pdfborder={0 0 1},backref=false,colorlinks=false]
 {hyperref}

\makeatletter
\theoremstyle{plain}
\newtheorem{thm}{\protect\theoremname}[section]
\theoremstyle{plain}
\newtheorem{lem}[thm]{\protect\lemmaname}
\theoremstyle{remark}
\newtheorem{rem}[thm]{\protect\remarkname}
\theoremstyle{plain}
\newtheorem{prop}[thm]{\protect\propositionname}
\theoremstyle{plain}
\newtheorem{cor}[thm]{\protect\corollaryname}
\theoremstyle{definition}
\newtheorem{example}[thm]{\protect\examplename}

\@ifundefined{date}{}{\date{}}
\usepackage{prettyref}

\usepackage{enumerate}
\allowdisplaybreaks

\newcommand{\R}{\mathbb{R}}
\newcommand{\N}{\mathbb{N}}
\newcommand{\dom}{\operatorname{dom}}
\newcommand{\ran}{\operatorname{ran}}

\renewcommand{\Re}{\operatorname{Re}}

\renewcommand{\tilde}{\widetilde}

\newcommand{\diag}{\operatorname{diag}}

\newrefformat{subsec}{Subsection \ref{#1}}
\newrefformat{prob}{Problem \ref{#1}}
\newrefformat{prop}{Proposition \ref{#1}}
\newrefformat{lem}{Lemma \ref{#1}}
\newrefformat{thm}{Theorem \ref{#1}}
\newrefformat{cor}{Corollary \ref{#1}}
\newrefformat{rem}{Remark \ref{#1}}
\newrefformat{exa}{Example \ref{#1}}
\newrefformat{sub}{Subsection \ref{#1}}
\newrefformat{eq}{(\ref{#1})}

\theoremstyle{definition}

\newrefformat{hyp}{Hypotheses \ref{#1}}

\makeatother

\providecommand{\corollaryname}{Corollary}
\providecommand{\examplename}{Example}
\providecommand{\lemmaname}{Lemma}
\providecommand{\propositionname}{Proposition}
\providecommand{\remarkname}{Remark}
\providecommand{\theoremname}{Theorem}

\begin{document}
\title{Adjoints of sums of m-accretive operators and applications to non-autonomous
evolutionary equations}
\author{Rainer Picard\thanks{Institut für Analysis, TU Dresden, Germany},
Sascha Trostorff\thanks{Mathematisches Seminar, CAU Kiel, Germany},
and Marcus Waurick\thanks{Institut für Angewandte Analysis, TU Bergakademie Freiberg, Germany}}
\maketitle
\begin{abstract}
We provide certain compatibility conditions for m-accretive operators
such that the adjoint of the sum is given by the closure of the sum
of the respective adjoint. We revisit the proof of well-posedness
of the abstract class of partial differential-algebraic equations
known as evolutionary equations. We show that the general mechanism
provided here can be applied to establish well-posedness for non-autonomous
evolutionary equations with $L_{\infty}$-coefficients thus not only
generalising known results but opening up new directions other methods
such as evolution families have a hard time to come by.
\end{abstract}
\label{=00005Cnoindent}\textbf{MSC2020 }35P05 (Primary), 47D99, 35Q61,
35Q59, 35K05, 35L05 (Secondary)

\tableofcontents{}

\section{Introduction}

Evolutionary Equations as introduced in the seminal paper \cite{Picard2009}
provide a Hilbert space perspective towards numerous (both linear
and non-linear) time-dependent phenomena in mathematical physics.
We refer to the monographs \cite{Primer,EvoEq} for a set of examples
as well as further development of the theory. It is instrumental for
the success of the theory of evolutionary equations that many (if
not all) equations from mathematical physics can be written as a time-dependent
partial differential-algebraic equation. Then, establishing the time-derivative
as an m-accerive, normal operator in some weighted Hilbert space and
gathering all the other unbounded operators (i.e., spatial derivative
operators) in an abstract m-accretive operator $A$ defined on some
Hilbert space enconding the spatial variables, one can write evolutionary
equation as an operator equation in the following form

\begin{equation}
\left(\partial_{0}\mathcal{M}_{0}+\mathcal{M}_{1}+A\right)U=F,\label{eq:evo0}
\end{equation}
where $\partial_{0}$ is the time-derivative, $U$ is the unknown,
$F$ models external forces and $\mathcal{M}_{0}$ and $\mathcal{M}_{1}$
are linear operators in the considered space-time Hilbert space, which
is a tensor product Hilbert space putting together temporal and spatial
variables. Any standard solution theory for evolutionary equations
of the form \prettyref{eq:evo0} provides conditions on the so-called
material law oparators, $\mathcal{M}_{0}$ and $\mathcal{M}_{1}$,
so that 
\[
\overline{\left(\partial_{0}\mathcal{M}_{0}+\mathcal{M}_{1}+A\right)}-c
\]
becomes m-accretive in the space-time Hilbert space. Looking into
the different proofs under the various assumptions on the material
law operators, one realises that the principal mechanism of showing
m-accretivity is based on the following general well-known fact. Quickly
recall that an operator $T$ in some Hilbert space is \textbf{accretive}\footnote{We assume every Hilbert space to be real.},
if for all $u\in\dom(T),$
\[
\langle u,Tu\rangle\geq0.
\]
$T$ is \textbf{m-accretive}, if $T$ is accretive and $T+\lambda$
is onto for all $\lambda>0$ (or equivalently for some $\lambda>0$).
\begin{thm}[{see also \cite[Chapter 3, Theorem 1.43]{Hu}}]
Let $H$ be a Hilbert space, $T\colon\dom(T)\subseteq H\to H$ densely
defined and closed linear operator. Then the following conditions
are equivalent:
\begin{enumerate}
\item $T$ is m-accretive;
\item $T$ and $T^{*}$ are accretive.
\end{enumerate}
If $T-c$ is m-accretive for some $c>0$, then $0\in\rho(T)$.
\end{thm}

In order to apply the last theorem showing m-accretivity of $T\coloneqq\overline{\left(\partial_{0}\mathcal{M}_{0}+\mathcal{M}_{1}+A\right)}-c$,
it is necessary to compute the adjoint of $T$, which, applying standard
results, boils down to computing the adjoint of $\left(\partial_{0}\mathcal{M}_{0}+\mathcal{M}_{1}+A\right)$
in the space-time Hilbert space. Since both $\partial_{0}\mathcal{M}_{0}$
and $A$ are generically speaking unbounded operators, this is a non-trivial
task. We refer to \cite{MHM} and the references therein for a general
account of computing adjoints of sums of unbounded operators (focussing
on situations complemented to the present one here). In any case,
once a formula of the type 
\[
\left(\partial_{0}\mathcal{M}_{0}+\mathcal{M}_{1}+A\right)^{*}=\overline{\left(\left(\partial_{0}\mathcal{M}_{0}+\mathcal{M}_{1}\right)^{*}+A^{*}\right)}
\]
is established, the accretivity of $T$ (\emph{and }of $T^{*}$) follow
from m-accretivity of $\partial_{0}\mathcal{M}_{0}+\mathcal{M}_{1}$.
Thus, conditions for \prettyref{eq:evo0} being well-posed need to
address the two facts: computing the adjoint in the way sketched above
needs to be possible and the problem needs to be m-accretive if $A=0$. 

The aim of this article is to understand the situation for the case
$\mathcal{M}_{0}=M_{0}(m_{0})$ and $\mathcal{M}_{1}=M_{1}(m_{0})$
are multiplication operators of multiplying in the time-variable by
$t\mapsto M_{0}(t)$ and $t\mapsto M_{1}(t)$, respectively. It is
known that Lipschitz continuous $M_{0}$ allows for computing the
adjoint as above and suitable positive definiteness conditions for
$M_{0}$ together with its (a.e. existing) derivative $M_{0}'$ and
$M_{1}$ lead to m-accretivity for the case $A=0$, see \cite{Picard2013}
or \cite[Chapter 16]{EvoEq}. 

A particular instance of the perspective of using operator sums to
understand partial differential equations has been provided (at least)
as early as \cite{daPrato1975}. However, the methods fail to apply
in a straightforward manner as the coefficient $M_{0}$ is allowed
to have a non-trivial kernel here (at least in the case of Lipschitz
continuous $M_{0}$). We illustrate our findings in the non Lipschitz
case by means of an example later on; note that this particular instance
was addressed in \cite{budde2023}. Even though we were not able to
fully rectify the arguments mentioned in this reference, their major
application is concerned with continuous-in-time coefficients anyway.
It seems that this condition is crucial for evolution families to
be applicable. Thus, in case of non-Lipschitz continuous $M_{0}$,
we establish well-posedness for a system of equations other methods
(such as evolution families or evolution semigroups, see \cite[Chapter VI, Section 9]{Engel_Nagel}
and \cite{Nickel1997}) are structurally deemed to fail.

The next section is concerned with some functional analytic preliminaries,
which we shall find useful in the subsequent parts. The subsequent
section contains our main result concerning operator sums of m-accretive
operators. The second to last section deals with applications to evolutionary
equations. We summarise our findings and open problems in the conclusion
section.

\section{Preliminaries}

Throughout this section, let $H_{0},H_{1},H_{2}$ be Hilbert spaces. 
\begin{lem}
\label{lem:BTclosed}Let $T\colon\dom(T)\subseteq H_{0}\to H_{1}$
be a closed linear operator and $B\colon H_{1}\to H_{2}$ bounded
and linear. Then, if $B$ is one-to-one and has closed range, $BT$
is closed.
\end{lem}

\begin{proof}
The closed graph theorem yields that the adstriction $\tilde{B}\colon H_{1}\to\ran(B)$
of $B$ is a continuously invertible operator. Next, let $(x_{n})_{n}$
in $\dom(BT)$ such that $x_{n}\to x$ and $BTx_{n}\to y$ in $H_{0}$
and $H_{2}$ as $n\to\infty$ for some $x\in H_{0}$ and $y\in H_{2}$.
Then, by the continuity of $(\tilde{B})^{-1}$, we infer $Tx_{n}=\left(\tilde{B}\right)^{-1}BTx_{n}\to\left(\tilde{B}\right)^{-1}y$
in $H_{1}$ as $n\to\infty$. By the closedness of $T$, we obtain
$x\in\dom(T)\subseteq\dom(BT)$ and $Tx=\left(\tilde{B}\right)^{-1}y.$
Applying $\tilde{B}$ to both sides of the latter equality, we infer
$y=BTx$ as desired.
\end{proof}
\begin{thm}
\label{thm:adjointTB}Let $T\colon\dom(T)\subseteq H_{0}\to H_{1}$
be closed, $B\colon H_{2}\to H_{0}$ be a bounded linear operator.
Then, as an identity of relations, we have
\[
(TB)^{*}=\overline{B^{*}T^{*}}.
\]
Moreover, $TB$ is densely defined if and only if $B^{*}T^{*}$ is
a closable operator. If, $B^{*}$ is one-to-one and has closed range
and $T$ is densely defined, then $B^{*}T^{*}$is closed. In particular,
in this case, we have $TB$ is densely defined and
\[
(TB)^{*}=B^{*}T^{*}.
\]
\end{thm}

\begin{proof}
The first statement is a consequence of \cite[Theorem 2.3.4]{EvoEq}.
The stated equivalence follows from $(TB)^{*}=\overline{B^{*}T^{*}}$
in conjunction with \cite[Lemma 2.2.7]{EvoEq}. Finally, the statement
containing $B^{*}T^{*}$ closed is a direct consequence of \prettyref{lem:BTclosed}.
\end{proof}
\begin{rem}
\label{rem:Bonto}The assumptions on $B^{*}$ are equivalent to $B$
being onto. Indeed, by the closed range theorem, the closed range
of $B$ is then inherited by $B^{*}$ and the equation $H_{0}=\ker(B^{*})\oplus\overline{\ran}(B)$
yields that (almost) surjectivity of $B$ is equivalent to injectivity
of $B^{*}$.
\end{rem}

\begin{thm}
\label{thm:adjointB*TB}Let $T\colon\dom(T)\subseteq H_{0}\to H_{0}$
be densely defined and closed, \textbf{$B\colon H_{0}\to H_{0}$}
be bounded, linear operator mapping onto $H_{0}$. Then
\[
\left(B^{*}TB\right)^{*}=B^{*}T^{*}B.
\]
\end{thm}

\begin{proof}
By \prettyref{lem:BTclosed} and \prettyref{rem:Bonto}, $B^{*}T$
is closed. It is -- trivially -- densely defined as so is $T$.
Hence, by \prettyref{thm:adjointTB}, we deduce
\[
\left((B^{*}T)B\right)^{*}=\overline{B^{*}(B^{*}T)^{*}}.
\]
Next, since $\left(B^{*}T\right)^{*}$ is a closed linear operator
as the adjoint of a densely defined operator and as $B^{*}$ is one-to-one
with closed range (see \prettyref{rem:Bonto}), $B^{*}(B^{*}T)^{*}$
is closed by \prettyref{lem:BTclosed}; i.e., $\overline{B^{*}(B^{*}T)^{*}}=B^{*}(B^{*}T)^{*}$.
Next, we compute $(B^{*}T)^{*}$. For this, using \prettyref{thm:adjointTB}
again, we deduce
\[
\left(T^{*}B\right)^{*}=\overline{B^{*}T^{**}}=\overline{B^{*}\overline{T}}=\overline{B^{*}T}=B^{*}T,
\]
as $B^{*}T$ is closed, again by \prettyref{lem:BTclosed}. Computing
adjoints on both sides, we infer
\[
T^{*}B=\left(T^{*}B\right)^{**}=\left(B^{*}T\right)^{*},
\]
where we used that $T^{*}B$ is closed. As a consequence of the above,
we get
\[
\left((B^{*}T)B\right)^{*}=B^{*}(B^{*}T)^{*}=B^{*}T^{*}B.\tag*{{\qedhere}}
\]
\end{proof}
\begin{prop}
Let $T\colon\dom(T)\subseteq H_{0}\to H_{0}$ be densley defined and
closed and $B\colon H_{0}\to H_{0}$ be a topological isomorphism.
If $D\subseteq\dom(T)$ is a core for $T$, then $B^{-1}[D]$ is a
core for $B^{*}TB.$
\end{prop}

\begin{proof}
$B$ being a topological isomorphism, $B^{-1}$ maps dense sets onto
dense sets; thus $B^{-1}[D]$ is dense in $H_{0}$. Also it is elementary
to see that $B^{-1}[D]\subseteq\dom(B^{*}TB)$. Finally, let $x\in\dom(B^{*}TB)=\dom(TB).$
Then $Bx\in\dom(T).$ By assumption, we find $(y_{n})_{n}$ in $D$
such that $y_{n}\to Bx$ and $Ty_{n}\to TBx$ in $H_{0}$ as $n\to\infty$.
Defining $x_{n}\coloneqq B^{-1}y_{n}\in B^{-1}[D]$, we get $x_{n}\to B^{-1}Bx=x$
and $TBx_{n}=Ty_{n}\to TBx$ as $n\to\infty.$ The continuity of $B^{*}$
yields the assertion.
\end{proof}

\section{Adjoints of Sums of m-accretive Operators}

This section is devoted to computing the adjoint of a sum of two (unbounded)
operators $T$ and $S$ both densely defined and closed on a Hilbert
space $H$. The aim is to provide conditions so that
\[
(S+T)^{*}=\overline{{S^{*}+T^{*}}}.
\]
We refer to \cite{Picard2022}, where several conditions for this
equality were given. The main theorem, which in turn is relevant to
the applications we have in mind, reads as follows:
\begin{thm}
\label{thm:adjoint_sum_pre}Let $T\colon\dom(T)\subseteq H_{0}\to H_{1}$
and $S\colon\dom(S)\subseteq H_{0}\to H_{1}$ be two densely defined
closed operators such that $\dom(S)\cap\dom(T)$ is dense. Moreover,
we assume that there exist families $(L_{\varepsilon})_{\varepsilon>0}$
$\left(K_{\varepsilon}\right)_{\varepsilon>0}$ and $(\tilde{K}_{\varepsilon})_{\varepsilon>0}$
in $L(H_{1})$ and $\left(R_{\varepsilon}\right)_{\varepsilon>0}$,
in $L(H_{0})$ such that $L_{\varepsilon}\to1_{H_{1}},\,R_{\varepsilon}\to1_{H_{0}}$
and $K_{\varepsilon},\tilde{K}_{\varepsilon}\to0$ in the weak operator
topology as $\varepsilon\to0.$ Moreover, assume that 
\begin{align}
L_{\varepsilon}S & \subseteq SR_{\varepsilon}+K_{\varepsilon},\nonumber \\
L_{\varepsilon}T & \subseteq TR_{\varepsilon}+\tilde{K}_{\varepsilon}\label{eq:trans}
\end{align}
and 
\begin{equation}
L_{\varepsilon}^{\ast}[\dom\left((S+T)^{\ast}\right)]\subseteq\dom(S^{\ast})\cap\dom(T^{\ast}).\label{eq:reg}
\end{equation}
Then 
\[
(S+T)^{\ast}=\overline{S^{\ast}+T^{\ast}}.
\]
\end{thm}

\begin{proof}
Since in general $\overline{S^{\ast}+T^{\ast}}\subseteq(S+T)^{\ast}$
(see \cite[Theoerem 2.3.2]{EvoEq}) it suffices to prove the remaining
inclusion. So, let $u\in\dom(S+T)^{\ast}$ and set $u_{\varepsilon}\coloneqq L_{\varepsilon}^{\ast}u\in\dom(S^{\ast})\cap\dom(T^{\ast}).$
Let $v\in\dom(S)\cap\dom(T).$ We compute 
\begin{align*}
\langle(S^{\ast}+T^{\ast})u_{\varepsilon},v\rangle & =\langle L_{\varepsilon}^{\ast}u,(S+T)v\rangle\\
 & =\langle u,L_{\varepsilon}(S+T)v\rangle\\
 & =\langle u,(S+T)R_{\varepsilon}v+(K_{\varepsilon}+\tilde{K}_{\varepsilon})v\rangle\\
 & =\langle\left(R_{\varepsilon}^{\ast}(S+T)^{\ast}+(K_{\varepsilon}+\tilde{K}_{\varepsilon})^{\ast}\right)u,v\rangle
\end{align*}
and since $\dom(S)\cap\dom(T)$ is dense, we infer 
\[
(S^{\ast}+T^{\ast})u_{\varepsilon}=\left(R_{\varepsilon}^{\ast}(S+T)^{\ast}+(K_{\varepsilon}+\tilde{K}_{\varepsilon})^{\ast}\right)u\rightharpoonup(S+T)^{\ast}u,
\]
where we have used $R_{\varepsilon}^{\ast}\to1_{H_{0}}$ and $(K_{\varepsilon}+\tilde{K}_{\varepsilon})^{\ast}\to0$
in the weak operator topology. Since also $u_{\varepsilon}\rightharpoonup u$
(use again $L_{\varepsilon}^{\ast}\to1_{H_{1}}$ in the weak operator
topology), we infer that $u\in\dom\left(\overline{S^{\ast}+T^{\ast}}\right)$
with 
\[
\left(\overline{S^{\ast}+T^{\ast}}\right)u=(S+T)^{\ast}u.\tag*{\qedhere}
\]
 
\end{proof}
\begin{rem}
\label{rem:reg_for_free}If $SR_{\varepsilon}$ is bounded and $\dom(S)\cap\dom(T)$
is a core for $T$ in the above theorem, then \prettyref{eq:reg}
holds true. Indeed, from \prettyref{eq:trans} we infer 
\[
\left(SR_{\varepsilon}+K_{\varepsilon}\right)^{\ast}\subseteq\left(L_{\varepsilon}S\right)^{\ast}=S^{\ast}L_{\varepsilon}^{\ast}.
\]
If now $SR_{\varepsilon}$ is bounded, the operator on the left-hand
side in the above inclusion is bounded, and hence, the operator on
the right-hand side is defined on $H_{1}$, meaning that $\ran(L_{\varepsilon}^{\ast})\subseteq\dom(S^{\ast})$.
Hence, for $u\in\dom(T)\cap\dom(S)$ and $v\in\dom((S+T)^{\ast})$
we compute 
\begin{align*}
\langle Tu,L_{\varepsilon}^{\ast}v\rangle & =\langle(S+T)u-Su,L_{\varepsilon}^{\ast}v\rangle\\
 & =\langle L_{\varepsilon}(S+T)u,v\rangle-\langle u,S^{\ast}L_{\varepsilon}^{\ast}v\rangle\\
 & =\langle(S+T)R_{\varepsilon}u+(K_{\varepsilon}+\tilde{K}_{\varepsilon})u,v\rangle-\langle u,S^{\ast}L_{\varepsilon}^{\ast}v\rangle\\
 & =\langle u,R_{\varepsilon}^{\ast}(S+T)^{\ast}v+(K_{\varepsilon}+\tilde{K}_{\varepsilon})^{\ast}v-S^{\ast}L_{\varepsilon}^{\ast}v\rangle
\end{align*}
and since $\dom(T)\cap\dom(S)$ is a core for $T$, we infer that
also $L_{\varepsilon}^{\ast}v\in\dom(T^{\ast}).$
\end{rem}

\begin{thm}
\label{thm:adjointofsum}Let $S\colon\dom(S)\subseteq H\to H$ and
$T\colon\dom(T)\subseteq H\to H$ both m-accretive. If $(1+T)^{-1}(1+S)^{-1}=(1+S)^{-1}(1+T)^{-1}$,
then $\dom(S)\cap\dom(T)$ is dense in $H$ and
\[
(S+T)^{*}=\overline{{S^{*}+T^{*}}}.
\]
\end{thm}

Before we prove this result, we draw some elementary consequences
of the commutator condition. A first consequence of this will be that
$S+T$ is densely defined. 
\begin{prop}
\label{prop:commucore}Under the conditions of \prettyref{thm:adjointofsum},
the following holds:
\begin{enumerate}
\item For all $\varepsilon>0$, we have 
\[
(1+\varepsilon S)^{-1}T\subseteq T(1+\varepsilon S)^{-1}.
\]
\item $\dom(S)\cap\dom(T)$ is a core for $T$. In particular, $\dom(S)\cap\dom(T)$
is dense in $H$.
\end{enumerate}
\end{prop}

\begin{proof}
For the first statement, we observe that $(1+T)^{-1}(1+S)^{-1}=(1+S)^{-1}(1+T)^{-1}$
yields 
\[
(1+S)^{-1}(1+T)\subseteq(T+1)(1+S)^{-1}.
\]
As a consequence, 
\[
(1+S)^{-1}T\subseteq T(1+S)^{-1}.
\]
By \cite[Lemma 9.3.3 (a)]{EvoEq}, we deduce for all $\varepsilon>0$
that
\[
(1+\varepsilon S)^{-1}T=\frac{1}{\varepsilon}(\frac{1}{\varepsilon}+S)^{-1}T\subseteq T\frac{1}{\varepsilon}(\frac{1}{\varepsilon}+S)^{-1}=T(1+\varepsilon S)^{-1}.
\]

The second statement is based on the observation that $(1+\varepsilon S)^{-1}\to1$
as $\varepsilon\to0+$ in the strong operator topology (this follows
from the strong convergence on $\dom(S)$ and the uniform boundedness
of the resolvents). Let now $x\in\dom(T)$ and define $x_{\varepsilon}\coloneqq(1+\varepsilon S)^{-1}x.$
Then, $x_{\varepsilon}\to x$ as $\varepsilon\to0+$ and by part 1
of the present proposition, we deduce $x_{\varepsilon}\in\dom(T)$
and 
\[
Tx_{\varepsilon}=(1+\varepsilon S)^{-1}Tx\to Tx,
\]
yielding that $\dom(T)\cap\dom(S)$ is a core for $T$.
\end{proof}
\begin{proof}[Proof of \prettyref{thm:adjointofsum}]
 We apply \prettyref{thm:adjoint_sum_pre}. By \prettyref{prop:commucore}
we see that for $L_{\varepsilon}\coloneqq R_{\varepsilon}\coloneqq(1+\varepsilon S)^{-1}$
and $K_{\varepsilon}=\tilde{K}_{\varepsilon}=0$ the relations \prettyref{eq:trans}
are satisfied. Furthermore, by \prettyref{prop:commucore} we have
that $\dom(S)\cap\dom(T)$ is dense and a core for $T.$ Since clearly
$SR_{\varepsilon}$ is bounded, \prettyref{rem:reg_for_free} gives
that also \prettyref{eq:reg} is satisfied. Thus, the assertion follows
from \prettyref{thm:adjoint_sum_pre}. 
\end{proof}

\section{Applications to Evolutionary Equations}

This section is devoted to apply the previous findings to operator
equations in weighted, vector-valued $L_{2}$-type spaces. The general
setting can be found in \cite{Primer,EvoEq}. Throughout, let $H$
be a Hilbert space and for $\rho\in\R$ we let 
\[
L_{2,\rho}(\R;H)\coloneqq\{f\in L_{2,\textnormal{loc}}(\R;H);\int_{\R}\|f(t)\|_{H}^{2}\exp(-2\rho t)dt<\infty\},
\]
endowed with the obvious norm and corresponding scalar product. We
define 
\[
\partial_{0}\colon H_{\rho}^{1}(\R;H)\subseteq L_{2,\rho}(\R;H)\to L_{2,\rho}(\R;H),\phi\mapsto\phi'.
\]
For $\rho>0$, it can be shown that $\partial_{0}$ is m-accretive
with $\Re\partial_{0}=\rho$. 

For a bounded, strongly measurable, operator-valued function $M\colon\R\to L(H)$,
we denote by
\[
M(m_{0})\in L(L_{2,\rho}(\R;H))
\]
the associated multiplication operator of multiplying by $M$. In
applications, $M$ will be induced by scalar-valued measurable functions;
that is, $M\in L_{\infty}(\R).$ We will work under the following
standing assumptions.
\begin{enumerate}
\item Let $A\colon\dom(A)\subseteq H\to H$ be $m$-accretive.
\item Let $M_{0},M_{1}\colon\R\to L(H)$ be strongly measurable and uniformly
bounded.
\item $M_{0}(m_{0})^{*}=M_{0}(m_{0})$.
\end{enumerate}
Here we have employed the custom to re-use the notation $A$ for the
(canonically) extended operator defined on $L_{2,\rho}(\R;H)$ with
domain $L_{2,\rho}(\R;\dom(A))$. The aim is to study the well-posedness
of non-autonomous problems of the form 
\begin{equation}
\left(\partial_{0}M_{0}(m_{0})+M_{1}(m_{0})+A\right)U=F\label{eq:prob}
\end{equation}
under suitable commutator conditions of $M_{0}(m_{0})$ with $\partial_{0}$
or with $A$. 

\subsection*{Bounded Commutator with $\partial_{0}$}

We begin to study the case when $M_{0}(m_{0})$ and $\partial_{0}$
have a bounded commuator; that is, we assume there exists a strongly
measurable and uniformly bounded mapping $M'_{0}\colon\R\to L(H)$
such that 
\[
M_{0}(m_{0})\partial_{0}\subseteq\partial M_{0}(m_{0})-M_{0}'(m_{0}).
\]

\begin{rem}
In \cite{Picard2013} it was shown that this assumption is equivalent
to the Lipschitz-continuity of $M_{0}$. In this case, $M_{0}$ is
differentiable almost everywhere and $M_{0}'$ is just the so-defined
derivative of $M_{0}.$
\end{rem}

Moreover, we impose the following accretivity condition on $M_{0}(m_{0})$
and $M_{1}(m_{0})$: 
\begin{equation}
\exists c>0,\,\rho_{0}>0\,\forall\rho\geq\rho_{0}:\,\rho M_{0}(t)+\frac{1}{2}M_{0}'(t)+M_{1}(t)\geq c\quad(t\in\R\text{ a.e.}).\label{eq:pos_def_Lipschitz}
\end{equation}

\begin{lem}
\label{lem:accretivity_non_auto_Lipschitz}Assume \prettyref{eq:pos_def_Lipschitz}.
Then for each $\rho\geq\rho_{0}$ the operator
\[
\partial_{0}M_{0}(m_{0})+M_{1}(m_{0})-c
\]
is accretive. Moreover, $\dom(\partial_{0})$ is a core for this operator.
\end{lem}

\begin{proof}
If $u\in\dom(\partial_{0})$ we infer 
\begin{align*}
2\langle(\partial_{0}M_{0}(m_{0})+M_{1}(m_{0}))u,u\rangle & =\langle\left(\partial_{0}M_{0}(m_{0})+M_{0}(m_{0})\partial_{0}\right)u,u\rangle+\langle M_{0}'(m_{0})u,u\rangle+2\langle M_{1}(m_{0})u,u\rangle.
\end{align*}
Moreover, with $\partial_{0}^{*}=-\partial_{0}+2\rho$ (see \cite[Corollary 3.2.6]{EvoEq})
\begin{align*}
\langle\partial_{0}M_{0}(m_{0})u,u\rangle & =\langle u,M_{0}(m_{0})\partial_{0}^{\ast}u\rangle\\
 & =-\langle u,M_{0}(m_{0})\partial_{0}u\rangle+2\rho\langle u,M_{0}(m_{0})u\rangle,
\end{align*}
which gives 
\[
\langle\left(\partial_{0}M_{0}(m_{0})+M_{0}(m_{0})\partial_{0}\right)u,u\rangle=2\rho\langle u,M_{0}(m_{0})u\rangle.
\]
Summarising, we obtain 
\[
\langle(\partial_{0}M_{0}(m_{0})+M_{1}(m_{0}))u,u\rangle=\langle\left(\rho M_{0}(m_{0})+\frac{1}{2}M_{0}'(m_{0})+M_{1}(m_{0})\right)u,u\rangle\geq c\|u\|^{2}.
\]
It remains to prove that $\dom(\partial_{0})$ is a core for $\partial_{0}M_{0}(m_{0}).$
This however follows from 
\[
(1+\varepsilon\partial_{0})^{-1}M_{0}(m_{0})=M_{0}(m_{0})(1+\varepsilon\partial_{0})^{-1}-\varepsilon(1+\varepsilon\partial_{0})^{-1}M_{0}'(m_{0})(1+\varepsilon\partial_{0})^{-1},
\]
which gives 
\[
(1+\varepsilon\partial_{0})^{-1}\partial_{0}M_{0}(m_{0})=\partial_{0}M_{0}(m_{0})(1+\varepsilon\partial_{0})^{-1}-\varepsilon\partial_{0}(1+\varepsilon\partial_{0})^{-1}M_{0}'(m_{0})(1+\varepsilon\partial_{0})^{-1}.
\]
If now $u\in\dom(\partial_{0}M_{0}(m_{0}))$ we set $u_{\varepsilon}\coloneqq(1+\varepsilon\partial_{0})^{-1}u\in\dom(\partial_{0})$
and since $(1+\varepsilon\partial_{0})^{-1}\to1$ strongly, the latter
equality proves that $u_{\varepsilon}\to u$ with respect to the graph
norm of $\partial_{0}M_{0}(m_{0}).$ 
\end{proof}
We obtain \cite[Theoerem 16.3.1]{EvoEq} or the main result of \cite{Picard2013}
as a special case:
\begin{thm}
Assume \prettyref{eq:pos_def_Lipschitz}. Then for each $\rho\geq\rho_{0}$
the operator 
\[
\overline{\partial_{0}M_{0}(m_{0})+M_{1}(m_{0})+A}-c
\]
is m-accretive and hence, $\overline{\partial_{0}M_{0}(m_{0})+M_{1}(m_{0})+A}$
is boundedly invertible in $L_{2,\rho}(\R;H)$ yielding the well-posedness
of \prettyref{eq:prob}.
\end{thm}

\begin{proof}
It is clear that $\partial_{0}M_{0}(m_{0})+M_{1}(m_{0})+A-c$ is accretive
as it is the sum of two accretive operators. In order to show that
its closure is m-accretive, it suffices to show that its adjoint is
also accretive. For this, we compute its adjoint with the help of
\prettyref{thm:adjoint_sum_pre}. We set $S\coloneqq\partial_{0}M_{0}(m_{0})+M_{1}(m_{0})$
and $T\coloneqq A$. Then $C_{c}^{\infty}(\R;\dom(A))\subseteq\dom(S)\cap\dom(T)$
is dense in $L_{2,\rho}(\R;H)$ and it is even a core for $T.$ Setting
$L_{\varepsilon}\coloneqq(1+\varepsilon\partial_{0})^{-1}$, we obtain
\prettyref{eq:trans} with $R_{\varepsilon}=L_{\varepsilon},$ $\tilde{K}_{\varepsilon}=0$
and 
\[
K_{\varepsilon}=\varepsilon\partial_{0}(1+\varepsilon\partial_{0})^{-1}M_{0}'(m_{0})(1+\varepsilon\partial_{0})^{-1}.
\]
Finally, 
\[
SR_{\varepsilon}=\partial_{0}M_{0}(m_{0})\left(1+\varepsilon\partial_{0}\right)^{-1}=(1+\varepsilon\partial_{0})^{-1}\partial_{0}M_{0}(m_{0})+K_{\varepsilon}
\]
is bounded, and hence, \prettyref{eq:reg} holds by \prettyref{rem:reg_for_free}.
Thus, we can apply \prettyref{thm:adjoint_sum_pre} and obtain 
\[
\left(\partial_{0}M_{0}(m_{0})+M_{1}(m_{0})+A\right)^{\ast}=\overline{\left(\partial_{0}M_{0}(m_{0})+M_{1}(m_{0})\right)^{\ast}+A^{\ast}}.
\]
Since clearly $A^{\ast}$ is accretive, it remains to prove the strict
accretivity of $\left(\partial_{0}M_{0}(m_{0})+M_{1}(m_{0})\right)^{\ast}=\left(\partial_{0}M_{0}(m_{0})\right)^{\ast}+M_{1}(m_{0})^{\ast}$.
In order to compute the first adjoint we recall that $\dom(\partial_{0})$
is a core for $\partial_{0}M_{0}(m_{0})$ and hence 
\[
\left(\partial_{0}M_{0}(m_{0})\right)^{\ast}=\left(M_{0}(m_{0})\partial_{0}+M_{0}'(m_{0})\right)^{\ast}=\partial_{0}^{\ast}M_{0}(m_{0})+M_{0}'(m_{0}).
\]
Now, as in \prettyref{lem:accretivity_non_auto_Lipschitz} one proves
that \prettyref{eq:pos_def_Lipschitz} yields the accretivity of $\partial_{0}^{\ast}M_{0}(m_{0})+M_{0}'(m_{0})+M_{1}(m_{0})-c.$ 
\end{proof}

\subsection*{Commutator with $A$}

Here, we assume a commutator condition with $A$. To keep things simple,
we assume that there exists $d>0$ such that $M_{0}(t)\geq d$ for
almost every $t\in\R$ and that 
\begin{equation}
M_{0}(m_{0})A\subseteq AM_{0}(m_{0}).\label{eq:amcom}
\end{equation}

\begin{lem}
\label{lem:commusequality}Under the standing assumptions together
with \prettyref{eq:amcom}, we have
\[
M_{0}(m_{0})A=AM_{0}(m_{0}).
\]
\end{lem}

\begin{proof}
The inclusion $M_{0}(m_{0})A\subseteq AM_{0}(m_{0})$ leads to 
\[
M_{0}(m_{0})(A+1)\subseteq(A+1)M_{0}(m_{0}).
\]
Now, the right-hand side operator is one-to-one and the left-hand
side is onto. Hence, 
\[
M_{0}(m_{0})(A+1)=(A+1)M_{0}(m_{0}),
\]
which yields the assertion%
\end{proof}
\begin{rem}
It is a consequence of the definition of the square root (see also
\cite[Theoerem B.8.2 and its proof]{Primer}) that 
\[
M_{0}(m_{0})A\subseteq AM_{0}(m_{0})
\]
 leads to 
\[
M_{0}(m_{0})^{1/2}A\subseteq AM_{0}(m_{0})^{1/2};
\]
thus, by \prettyref{lem:commusequality}, it follows
\[
M_{0}(m_{0})^{1/2}A=AM_{0}(m_{0})^{1/2}.
\]
In particular, we obtain
\[
M_{0}(m_{0})^{-1/2}A=AM_{0}(m_{0})^{-1/2}.
\]
\end{rem}

For motivating the main result of this section, we consider the following
evolutionary equation
\[
\left(\partial_{0}M_{0}(m_{0})+M_{1}(m_{0})+A\right)U=F.
\]
After multiplication by $M_{0}(m_{0})^{1/2}$ the latter can be rewritten
\begin{align*}
 & \left(M_{0}(m_{0})^{1/2}\partial_{0}M_{0}(m_{0})^{1/2}+M_{0}(m_{0})^{1/2}M_{1}(m_{0})M_{0}(m_{0})^{-1/2}+M_{0}(m_{0})^{1/2}AM_{0}(m_{0})^{-1/2}\right)M_{0}(m_{0})^{1/2}U\\
 & =M_{0}(m_{0})^{1/2}F.
\end{align*}
Using the latter remark and \prettyref{eq:amcom}, we get
\[
\left(M_{0}(m_{0})^{1/2}\partial_{0}M_{0}(m_{0})^{1/2}+M_{0}(m_{0})^{1/2}M_{1}(m_{0})M_{0}(m_{0})^{-1/2}+A\right)M_{0}(m_{0})^{1/2}U=M_{0}(m_{0})^{1/2}F.
\]

\begin{rem}
\label{rem:m0before}If instead of the above equation, we consider
\[
\left(M_{0}(m_{0})\partial_{0}+M_{1}(m_{0})+A\right)U=F,
\]
we may multiply by $M_{0}(m_{0})^{-1/2}$ instead and obtain
\begin{align*}
 & \left(M_{0}(m_{0})^{1/2}\partial_{0}M_{0}(m_{0})^{1/2}+M_{0}(m_{0})^{-1/2}M_{1}(m_{0})M_{0}(m_{0})^{1/2}+M_{0}(m_{0})^{-1/2}AM_{0}(m_{0})^{1/2}\right)M_{0}(m_{0})^{-1/2}U\\
 & =M_{0}(m_{0})^{-1/2}F,
\end{align*}
which in turn leads to
\[
\left(M_{0}(m_{0})^{1/2}\partial_{0}M_{0}(m_{0})^{1/2}+M_{0}(m_{0})^{-1/2}M_{1}(m_{0})M_{0}(m_{0})^{1/2}+A\right)M_{0}(m_{0})^{-1/2}U=M_{0}(m_{0})^{-1/2}F,
\]
being basically of the same shape of equation as the one above with
$\partial_{0}M_{0}(m_{0})$.
\end{rem}

The main result of this section is the following.
\begin{thm}
\label{thm:adjointevoeq}Under the standing assumptions,
\begin{align*}
 & \left(M_{0}(m_{0})^{1/2}\partial_{0}M_{0}(m_{0})^{1/2}+M_{0}(m_{0})^{1/2}M_{1}(m_{0})M_{0}(m_{0})^{-1/2}+A\right)^{*}\\
 & =\overline{\left(M_{0}(m_{0})^{1/2}\partial_{0}^{*}M_{0}(m_{0})^{1/2}+M_{0}(m_{0})^{-1/2}M_{1}(m_{0})^{*}M_{0}(m_{0})^{-1/2}+A^{*}\right)}.
\end{align*}
\end{thm}

\begin{proof}
First of all note that $M_{0}(m_{0})^{1/2}M_{1}(m_{0})M_{0}(m_{0})^{-1/2}$
is a bounded linear operator and can, thus, be assumed to be $0$
when computing the adjoint. Next, 
\[
\left(M_{0}(m_{0})^{1/2}\partial_{0}M_{0}(m_{0})^{1/2}\right)^{*}=M_{0}(m_{0})^{1/2}\partial_{0}^{*}M_{0}(m_{0})^{1/2}
\]
by \prettyref{thm:adjointB*TB} applied to $T=\partial_{0}$ and $B=M_{0}(m_{0})^{1/2}=B^{*}$.
In particular, note that it particularly follows that 
\[
\left(M_{0}(m_{0})^{1/2}\partial_{0}M_{0}(m_{0})^{1/2}\right)
\]
is m-accretive. Thus, for proving the present theorem, it suffices
to apply \prettyref{thm:adjointofsum} to $T=M_{0}(m_{0})^{1/2}\partial_{0}M_{0}(m_{0})^{1/2}$
and $S=A$. What remains is to show the commutativity of the resolvents:
\begin{align*}
\left(1+T\right)^{-1}\left(1+S\right)^{-1} & =(1+M_{0}(m_{0})^{1/2}\partial_{0}M_{0}(m_{0})^{1/2})^{-1}(1+A)^{-1}\\
 & =(M_{0}(m_{0})^{1/2}(M_{0}(m_{0})^{-1}+\partial_{0})M_{0}(m_{0})^{1/2})^{-1}(1+A)^{-1}\\
 & =M_{0}(m_{0})^{-1/2}(M_{0}(m_{0})^{-1}+\partial_{0})^{-1}M_{0}(m_{0})^{-1/2}(1+A)^{-1}\\
 & =M_{0}(m_{0})^{-1/2}(M_{0}(m_{0})^{-1}+\partial_{0})^{-1}(1+A)^{-1}M_{0}(m_{0})^{-1/2}\\
 & =M_{0}(m_{0})^{-1/2}\left((1+A)(M_{0}(m_{0})^{-1}+\partial_{0})\right)^{-1}M_{0}(m_{0})^{-1/2}\\
 & =M_{0}(m_{0})^{-1/2}\left((M_{0}(m_{0})^{-1}+\partial_{0})(1+A)\right)^{-1}M_{0}(m_{0})^{-1/2}\\
 & =\left(1+S\right)^{-1}\left(1+T\right)^{-1},
\end{align*}
where we used $\left(1+A\right)\partial_{0}=\partial_{0}(1+A)$ and
\prettyref{lem:commusequality} for $M_{0}(m_{0})^{-1}(1+A)=(1+A)M_{0}(m_{0})^{-1}$.
Hence, \prettyref{thm:adjointofsum} is applicable and the assertion
follows.
\end{proof}
\begin{lem}
\label{lem:crho0}For each $c>0$ there exists $\rho_{0}>0$ such
that for each $\rho\geq\rho_{0}$ the operator 
\[
M_{0}(m_{0})^{1/2}\partial_{0}M_{0}(m_{0})^{1/2}+M_{0}(m_{0})^{1/2}M_{1}(m_{0})M_{0}(m_{0})^{-1/2}-c
\]
is m-accretive.
\end{lem}

\begin{proof}
For $u\in\dom(\partial_{0}M_{0}(m_{0})^{1/2})$ we compute 
\begin{align*}
 & \langle\left(M_{0}(m_{0})^{1/2}\partial_{0}M_{0}(m_{0})^{1/2}+M_{0}(m_{0})^{1/2}M_{1}(m_{0})M_{0}(m_{0})^{-1/2}\right)u,u\rangle\\
 & =\rho\|M_{0}(m_{0})^{1/2}u\|^{2}-\|M_{0}(m_{0})^{1/2}M_{1}(m_{0})M_{0}(m_{0})^{-1/2}\|\|u\|^{2}\\
 & \geq\left(\rho d-\|M_{0}(m_{0})^{1/2}M_{1}(m_{0})M_{0}(m_{0})^{-1/2}\|\right)\|u\|^{2}.
\end{align*}
Choosing now $\rho$ large enough, we infer the strict accretivity
of the operator. Since its adjoint is of the form 
\[
M_{0}(m_{0})^{1/2}\partial_{0}^{\ast}M_{0}(m_{0})^{1/2}+M_{0}(m_{0})^{-1/2}M_{1}(m_{0})^{\ast}M_{0}(m_{0})^{1/2}
\]
the same argument shows that for $\rho$ large enough, this operator
is also accretive, and hence, the assertion follows. 
\end{proof}
\begin{cor}
Under the standing assumptions, the operator
\[
\overline{\partial_{0}M_{0}(m_{0})+M_{1}(m_{0})+A}
\]
is boundedly invertible in $L_{2,\rho}(\R;H)$ for large enough $\rho.$ 
\end{cor}

\begin{proof}
In the present situation, consider 
\[
\tilde{T}\coloneqq\overline{T+S},
\]
where $T\coloneqq M_{0}(m_{0})^{1/2}\partial_{0}M_{0}(m_{0})^{1/2}+M_{0}(m_{0})^{1/2}M_{1}(m_{0})M_{0}(m_{0})^{-1/2}$
and $S\coloneqq A$. We will show that $\tilde{T}-c$ is m-accretive
for some $c>0$. By assumption and \prettyref{lem:crho0}, it is not
difficult to see that $\tilde{T}-c$ is accretive for some $c>0$
and all large enough $\rho>0$. Using the formula for the adjoint
in \prettyref{thm:adjointevoeq} and taking into account the accretivtiy
of $T^{*}-c$, we have that $\left(\tilde{T}\right)^{*}-c$ is, too,
accretive. Hence, $0\in\rho(\tilde{T})$. The reformulation just before
 \prettyref{rem:m0before} yields the assertion by multiplying $\tilde{T}$
by the topological isomorphism $M_{0}(m_{0})^{-1/2}$ from the left
and $M_{0}(m_{0})^{1/2}$ from the right.
\end{proof}
\begin{rem}
A similar result holds under the assumption that 
\[
M_{0}(m_{0})^{1/2}\partial_{0}M_{0}(m_{0})^{1/2}+M_{0}(m_{0})^{-1/2}M_{1}(m_{0})M_{0}(m_{0})^{1/2}-c
\]
 is m-accretive for some $c>0$. Then 
\[
0\in\rho\left(\overline{M_{0}(m_{0})\partial_{0}+M_{1}(m_{0})+A}\right).
\]
\end{rem}

\begin{example}
Let $d\in\N$ and consider $H\coloneqq L_{2}(0,1)^{d}.$ Moreover,
let $B\in\R^{d\times d}$ with $\|B\|\leq1$ and consider the operator
$A$ given by 
\begin{align*}
\dom(A) & \coloneqq\{u\in H^{1}(0,1)^{d}\,;\,u(0)=Bu(1)\},\\
Au & \coloneqq u'.
\end{align*}
Then $A$ is m-accretive by \cite[Theorem 4.1]{Picard2023}. Moreover,
let $c_{1},\ldots,c_{d}\in L_{\infty}(\R)$ such that $c_{1},\ldots,c_{d}\geq k>0$
almost everywhere. We set $M_{0}(t)\coloneqq\diag(c_{j}(t))$ and
assume that $BM_{0}(t)=M_{0}(t)B$ for almost every $t\in\R.$ Then
we clearly have for $u\in L_{2,\rho}(\R;\dom(A))$ that $M_{0}(m_{0})u\in L_{2,\rho}(\R;H^{1}(0,1)^{d})$
and that 
\[
BM_{0}(t)u(1)=M_{0}(t)Bu(1)=M_{0}(t)u(0);
\]
that is, $M_{0}(m_{0})u\in L_{2,\rho}(\R;\dom(A)).$ Moreover, 
\[
AM_{0}(m_{0})u=M_{0}(m_{0})u'=M_{0}(m_{0})Au,
\]
which shows $M_{0}(m_{0})A\subseteq AM_{0}(m_{0}).$ Hence, by our
findings above, the non-autonomous problem 
\[
\left(\partial_{0}M_{0}(m_{0})+A\right)u=f
\]
is well-posed in $L_{2,\rho}(\R;L_{2}(0,1)^{d})$ if we choose $\rho$
large enough. 
\end{example}

\section{Conclusion}

We provided applicable conditions for m-accretive operators so that
the adjoint of the sum can be represented as the closure of the sum
of adjoints of the individual operators. We applied this observation
to evolutionary equations and developed well-posedness criteria for
the same. The case of the operator $M_{0}$ being $L_{\infty}$-in
time only and having non-trivial (possibly time-independent) kernel
remains open for easily applicable conditions establishing well-posedness.
Thus, to properly address the general situation of time-dependent
partial differential algebraic equations appears to be a challenge
for future research.

\bibliographystyle{abbrvurl}

\end{document}